\documentclass[reqno, 12pt,a4letter]{amsart}
\usepackage{amsmath,amsxtra,amssymb,latexsym, amscd,amsthm}
\usepackage{graphicx,color}
\usepackage[utf8]{inputenc}
\usepackage[mathscr]{euscript}
\usepackage{mathrsfs}
\usepackage[english]{babel}
\usepackage{color}
\newtheorem {theorem}{Theorem}[section]
\newtheorem {corollary}{Corollary}[section]

\newtheorem {lemma}{Lemma}[section]

\newtheorem {example}{Example}[section]
\newtheorem {definition}{Definition}[section]
\newtheorem {remark}{Remark}[section]

\def\ar{a\kern-.370em\raise.16ex\hbox{\char95\kern-0.53ex\char'47}\kern.05em}
\def\ees{{\accent"5E e}\kern-.385em\raise.2ex\hbox{\char'23}\kern-.08em}
\def\eex{{\accent"5E e}\kern-.470em\raise.3ex\hbox{\char'176}}
\def\AR{A\kern-.46em\raise.80ex\hbox{\char95\kern-0.53ex\char'47}\kern.13em}
\def\EES{{\accent"5E E}\kern-.5em\raise.8ex\hbox{\char'23 }}
\def\EEX{{\accent"5E E}\kern-.60em\raise.9ex\hbox{\char'176}\kern.1em}
\def\ow{o\kern-.42em\raise.82ex\hbox{
  \vrule width .12em height .0ex depth .075ex \kern-0.16em \char'56}\kern-.07em}
\def\OW{O\kern-.460em\raise1.36ex\hbox{
\vrule width .13em height .0ex depth .075ex \kern-0.16em \char'56}\kern-.07em}
\def\UW{U\kern-.42em\raise1.36ex\hbox{
\vrule width .13em height .0ex depth .075ex \kern-0.16em \char'56}\kern-.07em}
\def\DD{D\kern-.7em\raise0.4ex\hbox{\char '55}\kern.33em}

\title{Tangencies and Polynomial Optimization}

\author{TI\EES N-S\OW N PH\d{A}M}
\address{Department of Mathematics, University of Dalat, 1 Phu Dong Thien Vuong, Dalat, Vietnam}
\email{sonpt@dlu.edu.vn}

\date{ \today}

\thanks{The author is partially supported by Vietnam National Foundation for Science and Technology Development (NAFOSTED)}

\begin{document}
\maketitle

\begin{abstract}
Given a polynomial function $f \colon \mathbb{R}^n \rightarrow \mathbb{R}$ and a unbounded basic closed semi-algebraic set $S \subset \mathbb{R}^n,$ 
in this paper we show that the conditions listed below are characterized exactly in terms of the so-called {\em tangency variety} of $f$ on $S$: 
\begin{itemize}
\item The $f$ is bounded from below on $S;$
\item The $f$ attains its infimum on $S;$
\item The sublevel set $\{x \in S \ | \ f(x) \le \lambda\}$ for $\lambda \in \mathbb{R}$ is compact;
\item The $f$ is coercive on $S.$
\end{itemize}
Besides, we also provide some stability criteria for boundedness and coercivity of $f$ on $S.$
\end{abstract}

\section{Introduction}

Let $f \colon \mathbb{R}^n \rightarrow \mathbb{R}$ be a polynomial function and $S$ a unbounded basic closed semi-algebraic subset of $\mathbb{R}^n.$ Consider the optimization problem
\begin{equation} \label{Problem}
\textrm{minimize } \ f(x) \quad \textrm{ for all } \quad x \in S. \tag{P}
\end{equation}
In this paper we are interested in the following questions:
\begin{enumerate}
\item When is $f$ bounded from below on $S?$
\item Suppose that $f$ is bounded from below on $S.$ When does the problem~\eqref{Problem} have a solution?
\item When is a sublevel set of the restriction of $f$ on $S$ compact?
\item When is $f$ coercive on $S?$
\item Suppose that $f$ is bounded from below on $S.$ Let $g \colon S \rightarrow \mathbb{R}$ be a continuous function. When is $f + g$ bounded from below on $S?$
\item Suppose that $f$ is coercive on $S.$ Let $g \colon S \rightarrow \mathbb{R}$ be a continuous function. When is $f + g$ coercive on $S?$
\end{enumerate}

These questions are not easy to answer. In fact, concerning the first question, Shor \cite{Shor1987} writes
\begin{quote}
``Checking that a given polynomial function is bounded from below is far from trivial.''
\end{quote}

Nie, Demmel, and Sturmfels in the paper \cite{Nie2006} (see also \cite{Demmel2007})  propose a method for finding the global infimum of a polynomial function via sum of squares relaxations under the assumption that the optimal value is attained; in the conclusion section of the paper,  
the authors write:
\begin{quote}
``This paper proposes a method for minimizing a multivariate polynomial $f(x)$
over its gradient variety. We assume that the infimum $f^*$ is attained. This
assumption is nontrivial, and we do not address the (important and difficult)
question of how to verify that a given polynomial $f(x)$ has this property.''
\end{quote}
Indeed, very recently, Ahmadi and Zhang \cite{Ahmadi2018} showed that the testing attainment of the optimal value of a polynomial optimization problem is strongly NP-hard.



It is well-known that Problem~\eqref{Problem} attains its optimal value provided that one of the following sufficient conditions holds:
\begin{itemize}
\item There is some $\lambda \in \mathbb{R}$ such that the sublevel set 
$\{x \in S \ | \ f(x) \le \lambda\}$ is nonempty compact.
\item The $f$ is coercive on $S.$
\end{itemize}
Again, both of these conditions are strongly NP-hard to test as shown in \cite[Section~3]{Ahmadi2018}.

In other lines of development, we also would like to mention that the coercivity of polynomials defined on basic closed semi-algebraic sets and its relation to the Fedoryuk and Malgrange conditions are analyzed by H\`a and Ph\d{a}m \cite{HaHV2010} (see also \cite{Kim2018}), while a sufficient condition for the coercivity of polynomials on $\mathbb{R}^n$ is provided by Jeyakumar, Lasserre, and Li \cite{Jeyakumar2014}. A connection between the coercivity of polynomials on $\mathbb{R}^n$ and their Newton polytopes is given by Bajbar and Stein \cite{Bajbar2015}. For coercive polynomials, the order of growth at infinity and how this relates to the stability of coercivity with respect to perturbations of the coefficients are considered by Bajbar and Stein \cite{Bajbar2018-1} and by Bajbar and Behrends \cite{Bajbar2018-2}.

In this paper, we show that the questions stated in the beginning of this section can be answered completely based on the information contained in the so-called {\em tangency variety} of $f$ on $S.$ 
It is worth noting that tangencies play an important role in solving numerically polynomial optimization problems, see the papers \cite{HaHV2008-2, HaHV2009-1} and the monograph \cite{HaHV2017} for more details.


The rest of this paper is organized as follows. Some definitions and preliminaries concerning optimality conditions and tangencies are presented in Section~\ref{Section2};
in particular, some properties of tangencies are new and are of interest by themselves. The main results are given in Section~\ref{Section3}. Finally, several examples are provided in Section~\ref{Section4}.

\section{Critical points and tangencies} \label{Section2}

\subsection{Preliminaries}
We start this section with some words about our notation. We suppose $1 \le n \in {\Bbb N}$ and abbreviate $(x_1, \ldots, x_n)$ by $x.$ The space $\mathbb{R}^n$ is equipped with the usual scalar product $\langle \cdot, \cdot \rangle$ and the corresponding Euclidean norm $\| \cdot\|.$ Let $\mathbb{B}_R := \{x \in \mathbb{R}^n \ | \ \|x\| \le R\}$ and  $\mathbb{S}_R := \{x \in \mathbb{R}^n \ | \ \|x\|  = R\}.$ By convention, the minimum of the empty set is $+\infty.$

Recall that a subset of $\mathbb{R}^n$ is called {\em semi-algebraic} if it is a finite union of sets of the form
$$\{x \in \mathbb{R}^n \ | \ f_i(x) = 0, i = 1, \ldots, k; f_i(x) > 0, i = k + 1, \ldots, p\}$$
where all $f_{i}$ are polynomials. A map $f \colon A \subset \mathbb{R}^n \to \mathbb{R}^m$ is said to be {\em semi-algebraic} if its graph
is a semi-algebraic subset in $\mathbb{R}^n \times \mathbb{R}^m.$

The class of semi-algebraic sets is closed under taking finite intersections, finite unions and complements; a Cartesian product of semi-algebraic sets is a semi-algebraic set. Moreover, a major fact concerning the class of semi-algebraic sets is its stability under linear projections; in particular, the closure and interior of a semi-algebraic set are semi-algebraic sets. For more details, we refer the reader to \cite{Bochnak1998} and \cite[Chapter~1]{HaHV2017}.

\subsection{Optimality conditions}

Throughout this paper, let $f, g_i, h_j  \colon \mathbb{R}^n \rightarrow \mathbb{R},$ $i = 1, \ldots, l, j = 1, \ldots, m,$ be polynomial functions and assume that the set
$$S :=\{ x \in \mathbb{R}^n \ | \ g_1(x) = 0, \ldots, g_l(x) = 0, \ h_1(x) \ge 0, \ldots, h_m(x) \ge 0\}$$
is nonempty and unbounded. It is well-known that the standard first-order necessary conditions for optimality in Problem~\eqref{Problem} are the following.

\begin{theorem}[Fritz-John optimality conditions] \label{FritzJohnOptimalityConditions}
If $x \in S$ is an optimal solution of Problem~\eqref{Problem},  then there exist real numbers $\kappa, \lambda_i, i = 1, \ldots, l,$ and $\nu_j, j = 1, \ldots, m,$ not all zero, such that
\begin{eqnarray*}
&&  \kappa \nabla  f({x})  - \sum_{i = 1}^{l} \lambda_i \nabla g_i({x}) - \sum_{j = 1}^{m} \nu_j \nabla h_j({x})  = 0, \\
&&  \nu_j h_j({x}) = 0, \ \nu_j \ge 0, \ \textrm{ for } j = 1, \ldots, m.
\end{eqnarray*}
\end{theorem}

Notice that if $\kappa = 0,$ the above conditions are not very informative about a minimizer and so, we usually make an assumption called a constraint qualification to ensure that $\kappa \ne 0.$ A constraint qualification--probably the one most often used in the design of algorithms--is defined as follows.

\begin{definition} {\rm
We say that the {\em linear independence constraint qualification} {\rm (LICQ)} holds on $S,$ if for every $x \in S,$ the set of vectors
$$\nabla g_i(x) \ \textrm{ and } \ \nabla h_j(x) \ \textrm{ for } \  i = 1, \ldots, l, \ j \in J(x),$$
is linearly independent, where $$J(x) := \{j \in \{1, \ldots, m\} \ | \ h_j(x) = 0 \}$$ is called the set of {\em active constraint indices.}
}\end{definition}

\begin{remark}{\rm
By the Sard theorem, it is not hard to show that the condition (LICQ) holds generically (see \cite{Bolte2018, HaHV2017, Spingarn1979}).
}\end{remark}

Under the assumption that {\rm (LICQ)} holds on $S,$ we may obtain the more informative optimality conditions due to Karush, Kuhn and Tucker (and called the KKT optimality conditions) where the real number $\kappa$ in Theorem~\ref{FritzJohnOptimalityConditions} can be taken to be $1.$

\begin{theorem}[KKT optimality conditions] \label{KKTOptimalityConditions}
Let {\rm (LICQ)} hold on $S.$ If  $x \in S$ be an optimal solution of Problem~\eqref{Problem}, then there exist real numbers $\lambda_i, i = 1, \ldots, l,$ and $\nu_j, j = 1, \ldots, m,$ such that
\begin{eqnarray*}
&&  \nabla  f({x})  - \sum_{i = 1}^{l} \lambda_i \nabla g_i({x}) - \sum_{j = 1}^{m} \nu_j \nabla h_j({x})  = 0, \\
&&  \nu_j h_j({x}) = 0, \ \nu_j \ge 0, \ \textrm{ for } j = 1, \ldots, m.
\end{eqnarray*}
\end{theorem}

The KKT optimality conditions lead to the following notion.

\begin{definition}{\rm
We define the {\em set of critical points of $f$ on $S$} to be the set:
\begin{eqnarray*}
\Sigma(f, S) \ :=\  \{x \in S  & | &  \textrm{there exist } \lambda_i,
\nu_j \in {\Bbb R} \textrm{ such that }
\\ && \nabla  f(x)  - \sum_{i = 1}^{l} \lambda_i \nabla g_i(x) -
\sum_{j = 1}^{m} \nu_j \nabla h_j(x) = 0, \\
&&  \nu_j h_j(x) = 0, j = 1, \ldots, m \}.
\end{eqnarray*}
}\end{definition}

\begin{remark}{\rm
(i) In the case $S = \mathbb{R}^n,$ we have
$$\Sigma(f, \mathbb{R}^n) = \{x \in \mathbb{R}^n \ | \ \nabla f(x) = 0 \},$$
which is the usual {\em set of critical points} of $f.$

(ii) In light of Theorem~\ref{KKTOptimalityConditions}, if {\rm (LICQ)} holds on $S,$ then every optimal solution of Problem~\eqref{Problem} belongs to $\Sigma(f, S).$ Moreover, we have:
}\end{remark}

\begin{lemma}\label{Lemma21}
If {\rm (LICQ)} holds on $S,$ then $f(\Sigma(f, S))$ is a finite set.
\end{lemma}
\begin{proof}
See \cite[Lemma~3.3]{HaHV2009-1}.
\end{proof}

\subsection{Tangencies}

Consider Problem \eqref{Problem}. By definition, we have
\begin{eqnarray*}
\inf_{x \in S} f(x) &\le& \inf_{x \in \Sigma(f, S)} f(x),
\end{eqnarray*}
and the inequality can be strict as shown in the following example.

\begin{example}\label{Example21} {\rm
Let $S := \mathbb{R}^2$ and $f(x, y) := (xy - 1)^2 + y^2.$ We have $f > 0$ on $\mathbb{R}^2$ and
\begin{eqnarray*}
f(k, \frac{1}{k})  &=& \frac{1}{k^2} \to 0 \quad \textrm{ as } \quad k \to \infty.
\end{eqnarray*}
Hence
\begin{eqnarray*}
\inf_{(x, y) \in \mathbb{R}^2} f(x, y) &=& 0.
\end{eqnarray*}
Note that
\begin{eqnarray*}
\nabla f(x, y) = (0, 0) \quad \Longleftrightarrow \quad (x, y) = (0, 0).
\end{eqnarray*}
Therefore,
\begin{eqnarray*}
\inf_{(x, y) \in \mathbb{R}^2} f(x, y) & = &  0 \ < \ 1 \ = \ f(0, 0).
\end{eqnarray*}
}\end{example}

Assume that Problem~\eqref{Problem} has no optimal solution. Then there exists a sequence $\{x^k\} \subset S$ such that
\begin{eqnarray*}
\lim_{k \to +\infty} \|x^k\| = +\infty  \quad & \textrm{ and } & \quad \lim_{k \to +\infty} f(x^k)  \ = \ \inf_{x \in S} f(x).
\end{eqnarray*}
Since the set $\{x \in S \ | \ \|x\|^2 = \|x^k\|^2 \}$ is nonempty compact, the optimization problem
\begin{eqnarray*}
\textrm{minimize } \ f(x) \quad \textrm{ subject to } \quad x \in  S \  \textrm{ and } \  \|x\|^2 = \|x^k\|^2
\end{eqnarray*}
has at least an optimal solution, say $y^k.$ In light of Theorem~\ref{FritzJohnOptimalityConditions}, then for each $k,$ there exist $\kappa, \lambda_i, \nu_j, \mu \in \mathbb{R},$ not all zero, such that
\begin{eqnarray*}
&& \kappa \nabla  f(y^k)  - \sum_{i = 1}^{l} \lambda_i \nabla g_i(y^k) - \sum_{j = 1}^{m} \nu_j \nabla h_j(y^k) - \mu y^k = 0,  \textrm{ and } \\
&&  \nu_j h_j({x}) = 0, \ \nu_j \ge 0, \ \textrm{ for } j = 1, \ldots, m.
\end{eqnarray*}
This observation leads to the following notion.

\begin{definition}\label{Definition23}
{\rm By the {\em tangency variety of $f$ on $S$} we mean the set
\begin{eqnarray*}
\Gamma(f, S) :=  \{x \in S &|&  \textrm{there exist $\kappa, \lambda_i, \nu_j, \mu \in \mathbb{R},$ not all zero, such that } \\
&& \kappa \nabla  f(x)  - \sum_{i = 1}^{l} \lambda_i \nabla g_i(x) - \sum_{j = 1}^{m} \nu_j \nabla h_j(x) - \mu x = 0,  \textrm{ and } \\
&&  \nu_j h_j(x) = 0, j = 1, \ldots, m  \}.
\end{eqnarray*}
}\end{definition}

Geometrically, the tangency variety $\Gamma(f, S)$ consists of all points $x \in S$ where the level sets of the restriction of $f$ on $S$ are tangent to the sphere in $\mathbb{R}^n$ centered in the origin with radius $\|x\|.$

\begin{remark}{\rm
(i) In the case $S = \mathbb{R}^n,$ we have
$$\Gamma(f, \mathbb{R}^n) = \{x \in \mathbb{R}^n \ | \ \textrm{rank} \begin{pmatrix} \nabla f(x)  \\ x \end{pmatrix} \le 1 \}.$$

(ii) It is easy to see that $\Sigma(f, S) \subset \Gamma(f, S).$ Furthermore, we have
\begin{eqnarray*}
\inf_{x \in S} f(x) & = & \inf_{x \in \Gamma(f, S)} f(x) \ \le \ \inf_{x \in \Sigma(f, S)} f(x).
\end{eqnarray*}
}\end{remark}

\begin{lemma}\label{Lemma22}
Assume that {\rm (LICQ)} holds on $S.$ Then there exists a real number $R > 0$ such that we have for all $x \in \Gamma(f, S) \setminus \mathbb{B}_R, $
\begin{eqnarray*}
&& \nabla  f(x) - \sum_{i = 1}^{l}\lambda_i \nabla g_i(x) - \sum_{j = 1}^{m}\nu_j \nabla h_j(x) - \mu x =
0, \textrm{ and } \\
&& \nu_j h_j(x) = 0, j = 1, \ldots, m,
\end{eqnarray*}
for some real numbers $\lambda_i, \nu_j,$ and $\mu.$
\end{lemma}
\begin{proof}
See \cite[Lemma~3.2]{HaHV2009-1}.
\end{proof}

Applying Hardt's triviality theorem \cite{Hardt1980} for the semi-algebraic function
$$\| \cdot\| \colon \Gamma(f, S) \rightarrow \mathbb{R}, \quad x \mapsto \|x\|,$$
we find a constant $R > 0$ such that the restriction
$$\| \cdot\| \colon \Gamma(f, S) \setminus \mathbb{B}_{R} \rightarrow (R, + \infty)$$
is a topological trivial fibration. Let $p$ be the  number of connected components of a fiber of this restriction. Then $\Gamma(f, S) \setminus \mathbb{B}_R$ has exactly $p$ connected components, say $\Gamma_1, \ldots, \Gamma_p,$ and each such component is a unbounded semi-algebraic set. Moreover,
for all $t > R$ and all $k = 1, \ldots, p,$ the sets $\Gamma_k \cap \mathbb{S}_t$ are connected. Corresponding to each $\Gamma_k,$ let $$f_k \colon (R, +\infty) \rightarrow \mathbb{R}, \quad t \mapsto f_k(t),$$
be the function defined by $f_k(t) :=  f(x),$ where $x \in \Gamma_k \cap \mathbb{S}_t.$

\begin{lemma}\label{Lemma23} 
Assume that {\rm (LICQ)} holds on $S.$ For all $R$ large enough, the following statements hold:
\begin{enumerate}
\item[{\rm (i)}]  All the functions $f_k$ are well-defined and semi-algebraic.
\item[{\rm (ii)}] Each the function $f_k$ is either constant or strictly monotone.
 \item[{\rm (iii)}] The function $f_k$ is constant if, and only if, $\Gamma_k \subset \Sigma(f, S).$
\end{enumerate}
\end{lemma}

\begin{proof}
We choose $R$ large enough so that the conclusion of Lemma~\ref{Lemma22} holds.

(i) Fix $k \in \{1, \ldots, p\}$ and take any $t > R.$ We will show that the restriction of $f$ on $\Gamma_k \cap \mathbb{S}_t$ is constant. To see this, let $\phi \colon [0, 1] \rightarrow {\Bbb R}^n$ be a smooth semi-algebraic curve such that $\phi(\tau) \in \Gamma_k \cap \mathbb{S}_t$ for all $\tau \in [0, 1].$ By definition, we have
\begin{eqnarray} \label{Eqn1}
\|\phi(\tau)\| \equiv t \quad & \textrm{ and } & \quad g_i(\phi(\tau)) \equiv 0, \ i = 1, \ldots, l.
\end{eqnarray}
Moreover, by Lemma~\ref{Lemma22}, there exists a semi-algebraic curve $(\lambda, \nu, \mu)  \colon [0, 1] \rightarrow {\Bbb R}^l \times \mathbb{R}^m \times \mathbb{R}$ such that
\begin{eqnarray}\label{Eqn2}
&& \nabla f(\phi(\tau))  -  \sum_{i = 1}^{l} \lambda_i(\tau) \nabla g_i(\phi(\tau)) - \sum_{j = 1}^{m} \nu_j(t) \nabla h_j(\phi(\tau)) - \mu(\tau) \phi(\tau) \equiv 0,\\
&& \nu_j(\tau) h_j(\phi(\tau)) \equiv 0, \ j = 1, \ldots, m. \label{Eqn3}
\end{eqnarray}

Since the functions $\nu_j$ and $h_j \circ \phi$ are semi-algebraic, it follows from the Monotonicity Lemma (see, for example, \cite[Theorem~1.8]{HaHV2017}) that there is a partition $0 =: \tau_1 < \cdots < \tau_{N} := 1$ of $[0, 1]$ such that on each interval ${(\tau_l, \tau_{l + 1})}$ these functions are smooth and either constant or strictly monotone, for $l \in \{1, \ldots, N - 1\}.$ Then, by \eqref{Eqn3}, we can see that either $\nu_j(\tau) \equiv 0$ or $(h_j \circ \phi)(\tau) \equiv  0$ on ${(\tau_l, \tau_{l + 1})}.$ In particular, we have
$$\nu_j(\tau) \frac{d}{ dt }(h_j \circ \phi)(\tau)  \equiv 0, \quad j = 1, \ldots, m.$$
It follows from \eqref{Eqn1}, \eqref{Eqn2}, and \eqref{Eqn3} that
\begin{eqnarray*}
\frac{d }{d\tau}(f \circ \phi)(\tau)
&=& \langle \nabla  f(\phi(\tau)), \frac{d \phi(\tau)}{d\tau} \rangle \\
&=& \sum_{i = 1}^{l} \lambda_i(\tau) \langle \nabla  g_i(\phi(\tau)), \frac{d
\phi(\tau)}{d\tau} \rangle + \sum_{j = 1}^{m} \nu_j(\tau) \langle \nabla  h_j(\phi(\tau)),
\frac{d \phi(\tau)}{d\tau} \rangle \\ &&  \hspace{5.1cm} + \ \mu(\tau) \langle \phi(\tau), \frac{d \phi(\tau)}{d\tau} \rangle \\
&=& \sum_{i = 1}^{l}  \lambda_i(\tau) \frac{d}{d\tau}(g_i \circ \phi)(\tau) + \sum_{j = 1}^{m}
\nu_j(\tau) \frac{d}{d\tau}(h_j \circ \phi)(\tau) + \mu(\tau) \frac{1}{2} \frac{d \|\phi(\tau)\|^2}{d\tau}\\
&=& 0.
\end{eqnarray*}
So $f$ is constant on the curve $\phi.$

On the other hand, since the set $\Gamma_k \cap \mathbb{S}_t$ is connected semi-algebraic, it is path connected. Hence, any two points in $\Gamma_k \cap \mathbb{S}_t$ can be joined by a piecewise smooth semi-algebraic curve (see \cite[Theorem~1.13]{HaHV2017}). It follows that  the restriction of $f$ on $\Gamma_k \cap \mathbb{S}_t$ is constant. Finally, by the Tarski--Seidenberg Theorem (see, for example, \cite[Theorem~1.5]{HaHV2017}), the function $f_k$ is semi-algebraic.

(ii) By increasing $R$ (if necessary) and applying the Monotonicity Lemma (see \cite[Theorem~1.8]{HaHV2017}), it is not hard to get this item.

(iii) {\em Necessity.}  We argue by contradiction: assume that the function $f_k$ is constant but there exists a point $x^* \in \Gamma_k \setminus \Sigma(f, S).$ Since the set $\Sigma(f, S)$ is closed and since the restriction $\| \cdot\| \colon \Gamma(f, S) \setminus \mathbb{B}_{R} \rightarrow (R, + \infty)$
is topological trivial fibration, we can find a sequence $\{x^{\ell}\}_{{\ell} \ge 1} \in \Gamma_k \setminus \Sigma(f, S)$ satisfying the following conditions:
\begin{enumerate}
\item[{\rm (a)}] $x^\ell$ tends to $x^*$ as $\ell$ tends to $+\infty;$ and
\item[{\rm (b)}] $\|x^\ell\| = \|x^*\| + \frac{1}{\ell}$ for all $\ell \ge 1.$
\end{enumerate}
By the Curve Selection Lemma (see \cite[Theorem~1.11]{HaHV2017}), there exists a smooth semi-algebraic curve
$(\phi, \lambda, \nu, \mu)  \colon [a, b] \rightarrow {\Bbb R}^l \times \mathbb{R}^m \times \mathbb{R},$
with $\phi(a) = x^*,$ such that for all $t \in [a, b],$ the following conditions hold:
\begin{eqnarray*}
&& \phi(t) \in \Gamma_k \setminus \Sigma(f, S), \quad \|\phi(t)\| \equiv t > R, \\
&& \nabla f(\phi(t))  -  \sum_{i = 1}^{l} \lambda_i(t) \nabla g_i(\phi(t)) - \sum_{j = 1}^{m} \nu_j(t) \nabla h_j(\phi(t)) - \mu(t) \phi(t) \equiv 0,\\
&& \nu_j(t) h_j(\phi(t)) \equiv 0, \ j = 1, \ldots, m.
\end{eqnarray*}
Note that the function $f \circ \phi$ is just $f_k,$ and so, it is constant (by the assumption). Then a simple calculation shows that
\begin{eqnarray*}
0 & = & \frac{d (f \circ \phi)(t)}{dt}  \ = \ \langle \nabla f(\phi(t)), \frac{d \phi(t) }{dt} \rangle \ = \ \mu(t)  \frac{1}{2}  \frac{d \|\phi(t)\|^2}{dt} \ = \ \mu(t) t.
\end{eqnarray*}
Hence, $\mu \equiv 0,$ and so the curve $\phi$ lies in $\Sigma(f, S),$ which is a contradiction. Therefore, $\Gamma_k \subset \Sigma(f, S).$

{\em Sufficiency.} 
As in the proof of \cite[Theorem~2.3]{HaHV2017}, we can see that the restriction of $f$ on each connected component of the set $\Sigma(f, S)$ is constant.
Hence, the function $f_k = f|_{\Gamma_k}$ is constant because $\Gamma_k$ is a connected set and $\Gamma_k \subset \Sigma(f, S).$ 
\end{proof}

By Lemma~\ref{Lemma24}, we have associated to the function $f$ a finite number of functions $f_k$ of a single variable.
As a consequence, we get the next corollary (see also \cite[Lemma~2.2]{HaHV2008-2}, \cite[Proposition~3.2]{HaHV2009-1}, and  \cite[Theorem~1.5]{TaLL1998}). Let
\begin{eqnarray*}
T_\infty(f, S) \ :=\ \{ \lambda \in {\Bbb R}
& | & \textrm{there exists a sequence $x^k \in \Gamma(f, S)$ such that}\\
&& \|x^k\| \rightarrow +\infty \quad \mathrm{ and } \quad  \ f(x^k) \rightarrow \lambda \},
\end{eqnarray*}
and we call it the {\em set of tangency values at infinity} of $f$ on $S.$

\begin{corollary}\label{Corollary21}
Assume that {\rm (LICQ)} holds on $S.$ We have
\begin{eqnarray*}
T_\infty(f, S) &  = & \big\{\lim_{t\to +\infty} f_k(t) \ | \ k = 1, \ldots, p \big\} \cap \mathbb{R}.
\end{eqnarray*}
In particular, the set $T_\infty(f, S)$ is finite.
\end{corollary}

\begin{proof}
Indeed, in light of Lemma~\ref{Lemma23}, there exist (finite or infinite) limits
$$\lim_{t\to +\infty} f_k(t)$$
for all $k = 1, \ldots, p.$ In particular, the set
$$\big\{\lim_{t\to +\infty} f_k(t) \ | \ k = 1, \ldots, p \big\}$$
is finite.

On the other hand, by the Curve Selection Lemma at infinity (see \cite[Theorem~1.12]{HaHV2017}), we can see that a real number $\lambda$ belongs to $T_\infty(f, S)$ if, and only if, there exists a smooth semi-algebraic curve $\phi \colon (R', +\infty) \rightarrow \mathbb{R}^n$ lying in $\Gamma(f, S)$ with $R' \ge R$ such that
$$\lim_{\tau \to +\infty} \| \phi(\tau)\| = +\infty \quad \textrm{ and } \quad \lim_{\tau \to +\infty} f(\phi(\tau)) = \lambda.$$
Increasing $R'$ if necessary, we may assume that the curve $\phi$ lies in $\Gamma_k$ for some $k \in \{1, \ldots, p\}.$ Consequently, we get $f(\phi(\tau))  = f_k(\|\phi(\tau)\|)$ for all $\tau > R'.$ Then the desired conclusion follows.
\end{proof}

\begin{remark}{\rm
It worth emphasizing that the finiteness of the set of tangency values at infinity plays an important role in solving numerically polynomial optimization problems, see \cite{HaHV2008-2, HaHV2009-1}. For more details on the subject, we refer the reader to the survey \cite{Laurent2009} and the monographs \cite{HaHV2017, Lasserre2009, Lasserre2015, Marshall2008} with the references therein.
}\end{remark}

\begin{corollary}\label{Corollary22}
Assume that {\rm (LICQ)} holds on $S.$ If $f$ is bounded from below on $S$ then
\begin{eqnarray*}
\inf_{x \in S} f(x) & = & \min\{ \lambda \ | \ \lambda \in f(\Sigma(f, S)) \cup T_\infty(f, S) \}.
\end{eqnarray*}
\end{corollary}

\begin{proof}
If $f$ attains its infimum on $S$ then $f_* := \inf_{x \in S} f(x)  \in f(\Sigma(f, S))$ because of Theorem~\ref{KKTOptimalityConditions}.
Otherwise, the argument given before Definition~\ref{Definition23} shows that $f_* \in T_\infty(f, S).$ In both cases, we have
\begin{eqnarray*}
f_* & \ge & \min\{ \lambda \ | \ \lambda \in f(\Sigma(f, S)) \cup T_\infty(f, S) \},
\end{eqnarray*}
from which follows the desired conclusion.
\end{proof}

For each $t > R,$ we have $S \cap \mathbb{S}_t$ is a nonempty compact semi-algebraic set. Hence, the function
$$\psi \colon (R, +\infty) \rightarrow \mathbb{R}, \quad t \mapsto \psi(t) := \min_{x \in S \cap \mathbb{S}_t} f(x),$$
is well-defined, and moreover, it is semi-algebraic because of the Tarski--Seidenberg Theorem (see, for example, \cite[Theorem~1.5]{HaHV2017}).

The following lemma is simple but useful.
\begin{lemma} \label{Lemma24}
For $R$ large enough, the following statements hold:
\begin{enumerate}
\item [{\rm (i)}] The functions $\psi$ and $f_1, \ldots, f_p$ are either coincide or disjoint.

\item [{\rm (ii)}] $\psi(t) = \min_{k = 1, \ldots, p} f_k(t)$ for all $t > R.$

\item [{\rm (iii)}] $\psi \equiv f_k$ for some $k \in \{1, \ldots, p\}.$
\end{enumerate}
\end{lemma}

\begin{proof}
(i) This is an immediate consequence of the Monotonicity Lemma (see, for example, \cite[Theorem~1.8]{HaHV2017}).

(ii) By construction, for all $t > R$ we have
\begin{eqnarray*}
\Gamma(f, S) \cap \mathbb{S}_t &=& \bigcup_{k = 1}^p \Gamma_k \cap \mathbb{S}_t.
\end{eqnarray*}
Therefore,
\begin{eqnarray*}
\psi(t)  &=& \min_{x \in S \cap \mathbb{S}_t} f(x) \ = \ \min_{x \in \Gamma(f, S) \cap \mathbb{S}_t} f(x) \ = \ \min_{k = 1, \ldots, p} \min_{x \in \Gamma_k \cap \mathbb{S}_t} f(x) \ = \ \min_{k = 1, \ldots, p}  f_k(t),
\end{eqnarray*}
where the second equality follows from Theorem~\ref{FritzJohnOptimalityConditions}.

(iii) This follows from Items (i) and (ii).
\end{proof}

In view of Lemma~\ref{Lemma23}, the functions $f_k, k = 1, \ldots, p,$ are either constant or strictly monotone. Consequently, the following limits exist:
$$\lambda_k := \lim_{t \to +\infty} f_k(t) \in \mathbb{R} \cup \{\pm \infty\} \quad \textrm{ for } \quad k = 1, \ldots, p.$$
Note that if $f_k \equiv \lambda_k,$ then $\lambda_k \in f(\Sigma(f, S)).$ Furthermore, by Lemma~\ref{Lemma24}, the limit $\lim_{t \to +\infty} \psi(t)$ exists and equals to $\lambda_k$ for some $k.$ 

\begin{lemma} \label{Lemma25}
We have
\begin{eqnarray*}
\lim_{t \to +\infty} \psi(t)  &=&  \min_{k = 1, \ldots, p} \lambda_k.
\end{eqnarray*}
\end{lemma}

\begin{proof}
Indeed, by Lemma~\ref{Lemma24}, $\psi(t) \le f_k(t)$ for all $t > R$ and all $k = 1, \ldots, p.$ Letting $t \to +\infty,$ we get
\begin{eqnarray}\label{Eqn4}
\lim_{t \to +\infty} \psi(t)  &\le& \min_{k = 1, \ldots, p} \lambda_k.
\end{eqnarray}
On the other hand, by Lemma~\ref{Lemma24} again, there exists an index $k \in \{1, \ldots, p\}$ such that $\psi \equiv f_k,$ and so
\begin{eqnarray*}
\lim_{t \to +\infty} \psi(t) = \lambda_k.
\end{eqnarray*}
Combining this with the inequality \eqref{Eqn4}, we get the desired conclusion.
\end{proof}

We finish this section with the following observation.
\begin{lemma} \label{Lemma26}
We have
\begin{eqnarray*}
\lim_{t \to +\infty} \psi(t)  &\ge& \inf_{x \in S} f(x)
\end{eqnarray*}
with the equality if $f$ does not attain its infimum on $S.$
\end{lemma}

\begin{proof}
Indeed, we have for all $t > R,$
\begin{eqnarray*}
\psi(t)  &=& \min_{x \in S \cap \mathbb{S}_t} f(x) \ \ge \ \inf_{x \in S} f(x).
\end{eqnarray*}
Letting $t \to +\infty,$ we get $\lim_{t \to +\infty} \psi(t) \ge \inf_{x \in S} f(x).$

Now suppose that $f$ does not attain its infimum on $S,$ then there exists a sequence $\{x^\ell\}_{\ell \ge 1} \subset S$ such that
$$\lim_{\ell \to +\infty} \|x^\ell\| = +\infty \quad \textrm{ and } \quad \lim_{\ell \to +\infty} f(x^\ell) = \inf_{x \in S} f(x).$$
On the other hand, by definition, it is clear that $\psi(\|x^\ell\|) \le f(x^\ell)$ for all $\ell$ large enough. Therefore,
$\lim_{t \to +\infty} \psi(t) \le \inf_{x \in S} f(x),$ and so the desired conclusion follows.
\end{proof}

Note that in the above lemma we do not assume that $f$ is bounded from below on $S.$

\section{Main results} \label{Section3}

In this section, we give some answers to the questions stated in the introduction section.
Recall that $f, g_i, h_j  \colon \mathbb{R}^n \rightarrow \mathbb{R},$ $i = 1, \ldots, l, j = 1, \ldots, m,$ are polynomial functions and that the set
$$S :=\{ x \in \mathbb{R}^n \ | \ g_1(x) = 0, \ldots, g_l(x) = 0, \ h_1(x) \ge 0, \ldots, h_m(x) \ge 0\}$$
is nonempty and unbounded. From now on we will assume that {\rm (LICQ)} holds on $S.$ 

Keeping the notations as in the previous section, we know that $\Gamma(f, S) \setminus \mathbb{B}_R$ has exactly $p$ connected components $\Gamma_1, \ldots, \Gamma_p,$ and each such component is a unbounded semi-algebraic set. Corresponding to each $\Gamma_k,$ the functions
$$f_k \colon (R, +\infty) \rightarrow \mathbb{R}, \quad t \mapsto f_k(t),$$
are defined. Also, recall that the function $\psi \colon (R, +\infty) \rightarrow \mathbb{R}$ is defined by
$$\psi(t) := \min_{x \in S \cap \mathbb{S}_t} f(x).$$
Here and in the following, $R$ is chosen large enough so that the conclusions of Lemmas~\ref{Lemma22}, \ref{Lemma23}, and \ref{Lemma24} hold.

\subsection{Boundedness}

In this subsection we present necessary and sufficient conditions for the boundedness from below and from above of the objective function $f$ on the feasible set $S.$

\begin{theorem} \label{Theorem31}
The following statements hold:
\begin{enumerate}
\item [{\rm (i)}]  $f$ is bounded from below on $S$ if, and only if, it holds that $\min_{k = 1, \ldots, p} \lambda_k  >  -\infty.$

\item [{\rm (ii)}] $f$ is bounded from above on $S$ if, and only if, it holds that $\max_{k = 1, \ldots, p} \lambda_k  <  +\infty.$

\item [{\rm (iii)}] $f$ is bounded neither from below nor from above if, and only if, it holds that
$$\min_{k = 1, \ldots, p} \lambda_k  =  -\infty \quad \textrm{ and } \quad \max_{k = 1, \ldots, p} \lambda_k  =  +\infty.$$
\end{enumerate}
\end{theorem}

\begin{proof}
We prove only Item~(i); the other items may be treated similarly.

In light of Lemma~\ref{Lemma21}, $f(\Sigma(f, S))$ is a finite subset of $\mathbb{R}.$ By Lemma~\ref{Lemma23}, for any $k \not \in K,$ we have $\lambda_k$ belongs to the set $f(\Sigma(f, S))$ and so it is finite. Combining this with Lemmas~\ref{Lemma25} and \ref{Lemma26}, we get the desired conclusion.
\end{proof}

In what follows we let
$$K :=\{k \ | \ f_k \textrm{ is not constant} \}.$$

\begin{remark}{\rm
By definition, the index set $K$ is empty if, and only if, the restriction of $f$ on $S$ is constant outside a compact set in $\mathbb{R}^n.$ Furthermore, in light of Lemma~\ref{Lemma23}, $K = \{1, \ldots, p\}$ if, and only if, the set $\Sigma(f, S)$ of critical points of $f$ on $S$ is (possibly empty) compact.
}\end{remark}

By the Growth Dichotomy Lemma~\cite[Lemma~1.7]{HaHV2017} and increasing $R$ if necessary, we can assume that each function $f_k, k \in K,$ is developed into a fractional power series of the form
\begin{eqnarray*}
f_k(t) &=& a_k t^{\alpha_k} + \textrm{ lower order terms in } t,
\end{eqnarray*}
where $a_k \in \mathbb{R} \setminus \{0\}$ and $\alpha_k \in \mathbb{Q}.$

\begin{theorem} \label{Theorem32}
With the above notation, the following statements hold:
\begin{enumerate}
\item [{\rm (i)}]  $f$ is bounded from below on $S$ if, and only if, for any $k \in K,$
$$\alpha_k > 0 \quad \Longrightarrow \quad a_k > 0.$$

\item [{\rm (ii)}] $f$ is bounded from above on $S$ if, and only if, for any $k \in K,$
$$\alpha_k > 0 \quad \Longrightarrow  \quad a_k < 0.$$

\item [{\rm (iii)}] $f$ is bounded neither from below nor from above if, and only if, there exist integer numbers $k, k' \in K$ such that
$$\alpha_k > 0, \quad a_k > 0, \quad \alpha_{k'} > 0, \quad a_{k'} < 0.$$
\end{enumerate}
\end{theorem}

\begin{proof}
We prove only Item~(i); the rest follows easily.

By Theorem~\ref{Theorem31}, $f$ is bounded from below on $S$ if, and only if, it holds that $\lambda_k = \lim_{t \to +\infty} f_k(t) >  -\infty$ for all $k \in K.$ Then Item~(i) follows immediately from the definition of $\alpha_k$ and $a_k.$
\end{proof}

\begin{remark}{\rm
Following \cite{Grandjean2007} and \cite{Kurdyka2000} we can say that the exponents $\alpha_k$ are {\em characteristic exponents of $f|_S$ at infinity at $\lambda_k.$}
}\end{remark}

\subsection{Existence of optimal solutions}

In this subsection we provide necessary and sufficient conditions for the existence of optimal solutions to the problem~\eqref{Problem}.
We start with the following result.

\begin{theorem} \label{Theorem33}
The function $f$ attains its infimum on $S$ if, and only if, it holds that
\begin{eqnarray*}
\Sigma(f, S) \ne \emptyset \quad \textrm{ and } \quad \min_{x \in \Sigma(f, S) } f(x) & \le & \min_{k \in K} \lambda_k.
\end{eqnarray*}
\end{theorem}

\begin{proof}
Note that $f(\Sigma(f, S))$ is a finite subset of $\mathbb{R}$ (see Lemma~\ref{Lemma21}).

{\em Necessity.}  Let $f$ attain its infimum on $S,$ i.e., there exists a point $x^* \in S$ such that
\begin{eqnarray*}
f(x^*) &=& \inf_{x \in S} f(x).
\end{eqnarray*}
In light of Theorem~\ref{KKTOptimalityConditions}, $x^* \in \Sigma(f, S)$ and so $\Sigma(f, S)$ is nonempty.

On the other hand, for all $t > R$ we have
\begin{eqnarray*}
\inf_{x \in S} f(x) & \le & \min_{x \in S \cap \mathbb{S}_t} f(x) \ = \ \psi(t) \ = \ \min_{k = 1, \ldots, p} f_k(t),
\end{eqnarray*}
where the last equality follows from Lemma~\ref{Lemma24}. Therefore,
$$f(x^*) \le f_k(t) \quad \textrm{ for } \quad k = 1, \ldots, p.$$
Letting $t \to +\infty,$ we get
\begin{eqnarray*}
f(x^*) & \le & \min_{k = 1, \ldots, p} \lambda_k \ \le \ \min_{k \in K} \lambda_k.
\end{eqnarray*}

{\em Sufficiency.}  By the assumption, we have
\begin{eqnarray*}
- \infty & < & \min_{x \in \Sigma(f, S)} f(x) \ = \  \min_{\lambda \in f(\Sigma(f, S))} \lambda \ \le \ \min_{k \in K} \lambda_k.
\end{eqnarray*}
It follows from Theorem~\ref{Theorem31} that $f$ is bounded from below on $S.$

Now, assume that $f$ does not attain its infimum on $S.$ Then
\begin{eqnarray*}
\min_{x \in \Sigma(f, S) } f(x) & > & \inf_{x \in S} f(x).
\end{eqnarray*}
Moreover, by Lemmas~\ref{Lemma25} and \ref{Lemma26}, we have
\begin{eqnarray*}
\inf_{x \in S} f(x) = \lim_{t \to +\infty} \psi(t) \ = \ \min_{k \in 1, \ldots, p} \lambda_k.
\end{eqnarray*}
Consequently,
\begin{eqnarray*}
\min_{x \in \Sigma(f, S)} f(x) & > & \min_{k \in 1, \ldots, p} \lambda_k.
\end{eqnarray*}
Thanks to Lemma~\ref{Lemma23}(iii), we know that $\lambda_k \in f(\Sigma(f, S))$ for all $k \not \in K.$ Therefore
\begin{eqnarray*}
\min_{x \in \Sigma(f, S)} f(x) & > &  \min_{k \in K} \lambda_k,
\end{eqnarray*}
which contradicts the assumption that $\min_{x \in \Sigma(f, S) } f(x) \le \min_{k \in K} \lambda_k.$
\end{proof}

\begin{corollary}\label{Corollary31}
The set of all optimal solutions of the problem $\inf_{s \in S} f(x)$ is nonempty compact if, and only if, it holds that
\begin{eqnarray*}
\Sigma(f, S) \ne \emptyset, \quad \min_{x \in \Sigma(f, S) } f(x) & \le & \min_{k \in K} \lambda_k,
\quad \textrm{ and } \quad \min_{x \in \Sigma(f, S) } f(x) \  < \  \min_{k \not \in K} \lambda_k.
\end{eqnarray*}
\end{corollary}

\begin{proof}
This is a direct consequence of Theorem~\ref{Theorem33} and Lemma~\ref{Lemma23}(iii).
\end{proof}

Recall that the set $T_\infty(f, S)$ of tangency values at infinity of $f$ on $S$ is a (possibly empty) finite set in $\mathbb{R}$ (see Corollary~\ref{Corollary21}). Furthermore, we have

\begin{theorem} \label{Theorem34}
Suppose that $f$ is bounded from below on $S.$ Then $f$ attains its infimum on $S$ if, and only if, it holds that
\begin{eqnarray*}
\Sigma(f, S) \ne \emptyset \quad \textrm{ and } \quad \min_{x \in \Sigma(f, S)} f(x) & \le & \min_{\lambda \in T_\infty(f, S)} \lambda.
\end{eqnarray*}
\end{theorem}

\begin{proof}
Indeed, by Lemma~\ref{Lemma23}(iii), $\lambda_k \in f(\Sigma(f, S))$ for all $k \not \in K.$ Consequently,
we obtain
\begin{eqnarray*}
\min_{x \in \Sigma(f, S)} f(x) & \le & \min_{k \in K} \lambda_k \quad \Longleftrightarrow \quad
\min_{x \in \Sigma(f, S)} f(x) \ \le \ \min_{k = 1, \ldots, p} \lambda_k.
\end{eqnarray*}
On the other hand, since $f$ is bounded from below on $S,$ it follows from Corollary~\ref{Corollary21} that
\begin{eqnarray*}
\min_{\lambda \in T_\infty(f, S)} \lambda &=& \min_{k = 1, \ldots, p} \lambda_k.
\end{eqnarray*}
Now, applying Theorem~\ref{Theorem33}, we get the desired conclusion.
\end{proof}

\subsection{Compactness of sublevel sets}

Recall that the function $\psi \colon (R, +\infty) \rightarrow \mathbb{R}, t \mapsto \psi(t),$ is defined by 
$$\psi(t) := \min_{x \in S \cap \mathbb{S}_t} f(x).$$
In view of Lemmas~\ref{Lemma23} and \ref{Lemma24}, the function $\psi$ is either constant or strictly monotone. Consequently, the following limit exists:
\begin{eqnarray*}
\lambda_* &:=& \lim_{t \to + \infty} \psi(t).
\end{eqnarray*}
We also note from Lemma~\ref{Lemma25} that
\begin{eqnarray*}
\lambda_* &=& \min_{k = 1, \ldots, p} \lambda_k.
\end{eqnarray*}

For each $\lambda \in \mathbb{R},$ we write 
\begin{eqnarray*}
\mathscr{L}(\lambda) &:=& \{x \in S \ | \ f(x) \le \lambda\}.
\end{eqnarray*} 
It is easy to see that if $\mathscr{L}(\lambda)$ is nonempty compact for some $\lambda,$ then the infimum $\inf_{x \in S} f(x)$ of $f$ on $S$ is finite and attained. 

\begin{theorem} \label{Theorem35}
Suppose that $f$ is bounded from below on $S.$ The following statements hold:
\begin{enumerate}
\item[{\rm (i)}] If $\lambda > \lambda_*$ then $\mathscr{L}(\lambda)$ is unbounded.
\item[{\rm (ii)}] If $\lambda < \lambda_*$ then $\mathscr{L}(\lambda)$ is compact.
\item[{\rm (iii)}] Assume that $\lambda = \lambda_*.$ Then $\mathscr{L}(\lambda)$ is compact if, and only if, the function
$\psi$ is strictly decreasing.
\end{enumerate}
\end{theorem}

\begin{proof}
(i) Assume that $\lambda > \lambda_* = \min_{k = 1, \ldots, p}\lambda_k.$ Then there exists an index $k \in \{1, \ldots, p\}$ such that $f_k(t) < \lambda$ for all $t$ large enough. By definition,
for each $t > R,$ there exists $\phi(t) \in \Gamma_k \cap \mathbb{S}_t$ such that $f_k(t) = f(\phi(t)).$
Hence $\phi(t) \in \mathscr{L}(\lambda)$ for sufficiently large $t,$ which yields $\mathscr{L}(\lambda)$ is unbounded.

(ii) Assume that $\lambda < \lambda_*.$ By contradiction, suppose that $\mathscr{L}(\lambda)$ is unbounded. Then for all sufficiently large $t,$ the set $\mathscr{L}(\lambda) \cap \mathbb{S}_t$ is not empty and it holds that
\begin{eqnarray*}
\psi(t) &=& \min_{x \in S \cap \mathbb{S}_t} f(x) \ = \ \min_{x \in \mathscr{L}(\lambda) \cap \mathbb{S}_t} f(x)  \ \le \ \lambda.
\end{eqnarray*}
This implies that
\begin{eqnarray*}
\lambda_* &=& \lim_{t \to + \infty} \psi(t) \ \le \ \lambda,
\end{eqnarray*}
which contradicts our assumption.

(iii) Assume that $\lambda = \lambda_*.$ We first assume that $\mathscr{L}(\lambda)$ is compact. Then, for all $t$ large enough we have
\begin{eqnarray*}
\psi(t) & = & \min_{x \in S \cap \mathbb{S}_t} f(x) \ > \ \lambda \ = \ \lambda_*,
\end{eqnarray*}
Since $\lim_{t \to + \infty} \psi(t) = \lambda_*,$ it follows that the function $\psi$ is strictly decreasing.

Conversely, assume that $\mathscr{L}(\lambda)$ is unbounded. Then for all $t$ large enough, the set $\mathscr{L}(\lambda) \cap \mathbb{S}_t$ is not empty, and so
\begin{eqnarray*}
\psi(t) & = & \min_{x \in S \cap \mathbb{S}_t} f(x) \ = \  \min_{x \in \mathscr{L}(\lambda) \cap \mathbb{S}_t} f(x)  \ \le \ \lambda \ = \  \lambda_*.
\end{eqnarray*}
Hence, the function $\psi$ is either constant $\lambda_*$ or strictly increasing.
\end{proof}

\subsection{Coercivity}

In this subsection we give necessary and sufficient conditions for the coercivity of $f$ on $S.$ Here and in the following, we say that $f$ is {\em coercive on} $S$ if for every sequence $\{x^k\} \subset S$ such that $\|x^k\| \to +\infty,$ we have $f(x^k) \to +\infty.$ It is well known that if $f$ is coercive on $S$ then all sublevel sets of $f$ on $S$ are compact, and so $f$ achieves its infimum on $S.$

\begin{theorem} \label{Theorem36}
The following statements are equivalent:

\begin{enumerate}
\item[{\rm (i)}] The function $f$ is coercive on $S.$
\item[{\rm (ii)}] $K = \{1, \ldots, p\}$ and $\alpha_k > 0$ and $a_k > 0$ for all $k = 1, \ldots, p.$
\item[{\rm (iii)}] $\lambda_* = + \infty.$
\item[{\rm (iv)}] The function $f$ is bounded from below on $S$ and $T_\infty(f, S) = \emptyset.$
\end{enumerate}
\end{theorem}

\begin{proof}
(i) $\Rightarrow$ (ii): By definition, $K = \{1, \ldots, p\}$ and $\alpha_k > 0$ for all $k = 1, \ldots, p.$ In view of Theorem~\ref{Theorem32}, then $a_k > 0$ for $k = 1, \ldots, p.$

(ii) $\Rightarrow$ (iii): This follows immediately from the definitions and the fact that $\lambda_* = \min_{k = 1, \ldots, p}\lambda_k.$

(iii) $\Rightarrow$ (iv): Since $\lim_{t \to + \infty} \psi(t) = \lambda_* = +\infty,$ it follows from Lemma~\ref{Lemma25} and Theorem~\ref{Theorem31} that $f$ is bounded from below on $S.$ Moreover, we have $T_\infty(f, S) = \emptyset,$ which follows from Corollary~\ref{Corollary21}.

(iv) $\Rightarrow$ (i): By contradiction, assume that $f$ is not coercive on $S.$ Then the limit
\begin{eqnarray*}
\lambda_* &:=& \lim_{t \to +\infty} \psi(t)
\end{eqnarray*}
is finite because $f$ is bounded from below on $S.$ On the other hand, in view of Lemma~\ref{Lemma25}, $\lambda_* = \lambda_k = \lim_{t \to +\infty}  f_k(t)$ for some $k \in \{1, \ldots, p\}.$ By Corollary~\ref{Corollary21}, then $\lambda_* \in T_\infty(f, S),$ which contradicts to the assumption that $T_\infty(f, S) = \emptyset.$
\end{proof}

\subsection{Stability}
In this subsection, we show some stability properties for semi-algebraic functions.

Given two numbers $\epsilon > 0$ and $\alpha \in \mathbb{R},$ let $\mathcal{F}_{\epsilon, \alpha}(S)$ denote the set of all functions $g \colon S \rightarrow \mathbb{R},$ for which there exists $R' > 0$ such that
$$|g(x)| \le \epsilon \|x\|^{\alpha} \quad \textrm{ for any } \quad x \in S \quad \textrm{ and } \quad \|x\| \ge R'.$$

\begin{remark}{\rm
Note that for any $\epsilon > 0$ and any  $\alpha \in \mathbb{R},$ the set $\mathcal{F}_{\epsilon, \alpha}(S)$ is nonempty. For example, it is easy to see that the function $g \colon S \rightarrow \mathbb{R}, x \mapsto e^{-\|x\|},$ belongs to any $\mathcal{F}_{\epsilon, \alpha}(S).$
}\end{remark}

Recall that for each $k \in K,$ we have asymptotically as $t \to +\infty,$
\begin{eqnarray*}
f_k(t) &=& a_k t^{\alpha_k} + \textrm{ lower order terms in } t,
\end{eqnarray*}
where $a_k \in \mathbb{R} \setminus \{0\}$ and $\alpha_k \in \mathbb{Q}.$ Let
\begin{eqnarray*}
\alpha_* &:=& \min_{k = 1, \ldots, p} \alpha_k,
\end{eqnarray*}
where $\alpha_k := 0$ for $k \not \in K.$ The following result gives a stability property for the boundedness from below of semi-algebraic functions.

\begin{theorem} \label{Theorem37}
Assume that $f$ is bounded from below on $S.$ The following two assertions hold:
\begin{enumerate}
\item[{\rm (i)}]  There exists $\epsilon > 0$ such that for all $\alpha \le \alpha_*$ and all $g \in \mathcal{F}_{\epsilon, \alpha}(S),$ the function $f + g$ is bounded from below on $S.$

\item[{\rm (ii)}] For all $\epsilon > 0$ and all $\alpha > \max\{0, \alpha_*\},$ there exists a semi-algebraic continuous function $g \in \mathcal{F}_{\epsilon, \alpha}(S)$ such that the function $f + g$ is not bounded from below on $S.$
\end{enumerate}
\end{theorem}

\begin{proof}
(i) The claim is clear in the case $\alpha_* \le 0.$ So assume that $\alpha_* > 0.$ Then all the functions $f_k$ are not constant, i.e., $K = 1, \ldots, p.$
By Theorem~\ref{Theorem32}, $\alpha_k > 0$ and $a_k > 0$ for all $k = 1, \ldots, p.$ Hence, there exist constants $c > 0$ and $R' > R$ such that for $k = 1, \ldots, p,$
\begin{eqnarray*}
f_k(t) &\ge& c\, t^{\alpha_*} \quad \textrm{ for all } \quad t \ge R'.
\end{eqnarray*}
Consequently, we have for all $x \in S$ with $\|x\| \ge R',$
\begin{eqnarray*}
f(x) &\ge& \min_{y \in S, \ \|y\| = \|x\|}f(y) \ = \ \psi(\|x\|) \ = \ \min_{k = 1, \ldots, p} f_k(\|x\|) \ \ge \ c\, \|x\| ^{\alpha_*}.
\end{eqnarray*}
Let $\epsilon := \frac{c}{2}.$ Take any $\alpha \le \alpha_*$ and let $g \colon S \rightarrow \mathbb{R}$ be a continuous function such that
$$|g(x)| \le \epsilon\, \|x\|^{\alpha} \quad \textrm{ for any } \quad x \in S \quad \textrm{ and } \quad \|x\| \ge R'.$$
We have for all $x \in S$ with $\|x\| \ge R',$
\begin{eqnarray*}
f(x) + g(x) &\ge& c\, \|x\| ^{\alpha_*} - \epsilon \, \|x\|^{\alpha}
\ \ge \ c \, \|x\| ^{\alpha_*} - \epsilon \,\|x\|^{\alpha_*}
\ = \ \epsilon \, \|x\|^{\alpha_*}.
\end{eqnarray*}
Clearly, this implies that the function $f + g$ is bounded from below on $S.$

(ii) Let $\epsilon > 0$ and $\alpha > \max\{0, \alpha_*\}.$ Take any rational number $\beta$ with $\alpha \ge \beta > \max\{0, \alpha_*\}.$
Define the function $g \colon S \rightarrow \mathbb{R}$ by $g(x) := - \epsilon \|x\|^\beta.$ Then $g$ is semi-algebraic continuous and belongs to $\mathcal{F}_{\epsilon, \alpha}(S).$

On the other hand, it is not hard to see that there exists a curve $\phi \colon (R, +\infty) \rightarrow \Gamma(f, S)$ such that $\| \phi(t) \| = t$ and asymptotically as $t \to +\infty,$
$$f(\phi(t)) = c\, t^{\alpha_*} + \textrm{ lower order terms in } t$$
for some $c \in \mathbb{R}.$ It follows that
\begin{eqnarray*}
f(\phi(t)) + g(\phi(t)) & = & - \epsilon\, t^{\beta}  + c\, t^{\alpha_*} + \textrm{ lower order terms in } t,
\end{eqnarray*}
which tends to $-\infty$ as $t$ tends to $+\infty.$ Hence, the function $f + g$ is not bounded from below on $S.$
\end{proof}

The next result gives a stability criterion for the coercivity of semi-algebraic functions.

\begin{theorem}
Assume that $f$ is coercive on $S.$ The following two assertions hold:
\begin{enumerate}
\item[{\rm (i)}] There exists $\epsilon > 0$ such that for all $\alpha \le \alpha_*$ and all $g \in \mathcal{F}_{\epsilon, \alpha}(S),$ the function $f + g$ is coercive on $S.$
\item[{\rm (ii)}] For all $\epsilon > 0$ and all $\alpha > \alpha_*,$ there exists a semi-algebraic continuous function $g \in \mathcal{F}_{\epsilon, \alpha}(S)$ such that the function $f + g$ is not coercive on $S.$
\end{enumerate}
\end{theorem}

\begin{proof}
(i) By Theorem~\ref{Theorem36}, we know that $K = \{1, \ldots, p\}$ and $\alpha_* > 0.$ Then the rest of the proof is analogous to that of Theorem~\ref{Theorem37}.

(ii) Let $\epsilon > 0$ and $\alpha > \alpha_*.$ Take any rational number $\beta$ with $\alpha \ge \beta > \alpha_*.$
Define the function $g \colon S \rightarrow \mathbb{R}$ by $g(x) := - \epsilon \|x\|^\beta.$ Clearly, $g$ is semi-algebraic continuous and belongs to $\mathcal{F}_{\epsilon, \alpha}(S).$ Moreover, as in the proof of Theorem~\ref{Theorem37}, we can see that the function $f + g$ is not bounded from below on $S,$ and so, it is not coercive on $S.$
\end{proof}

We finish this section by noting that it is not true that if $f$ attains its infimum on $S$ then there exists $\epsilon > 0$ such that for all $\alpha \le \alpha_*$ and all $g \in \mathcal{F}_{\epsilon, \alpha}(S),$ the function $f + g$ attains its infimum on $S.$

\begin{example}{\rm
Let $f(x, y) := x^2$ and $S := \mathbb{R}^2.$ Clearly, $f$ is bounded from below and attains its infimum on $S.$ A direct calculation shows that $\alpha_* = 0.$ Furthermore, for all $\epsilon > 0$ and all $\alpha \le 0,$ we have $g(x, y) := \epsilon (1 + \|(x, y)\|)^{\alpha - 1} \in \mathcal{F}_{\epsilon, \alpha}(S)$ and the function $f + g$ is bounded from below but does not attain its infimum on $S.$
}\end{example}

\section{Examples} \label{Section4}

In this section we provide examples to illustrate our main results. For simplicity we consider the case where $S := \mathbb{R}^2$ and $f$ is a polynomial function in two variables $(x, y) \in \mathbb{R}^2.$ By definition, then
\begin{eqnarray*}
\Sigma(f, \mathbb{R}^2)  &:=& \left \{(x, y) \in \mathbb{R}^2 \ | \ \frac{\partial f}{\partial x}  = \frac{\partial f}{\partial y} = 0 \right\}, \\
\Gamma(f, \mathbb{R}^2) &:=& \left \{(x, y) \in \mathbb{R}^2 \ | \ y \frac{\partial f}{\partial x} - x \frac{\partial f}{\partial y} = 0 \right\}.
\end{eqnarray*}

\begin{example}{\rm
Let $f(x, y) := x^3 - 3y^2.$ We have $\Sigma(f, \mathbb{R}^2)   = \{(0, 0)\},$ 
and the tangency variety $\Gamma(f, \mathbb{R}^2)$ is given by the equation:
\begin{eqnarray*}
3x^2y + 6xy &=& 0.
\end{eqnarray*}
Hence, for $R > 2,$ the set $\Gamma(f, \mathbb{R}^2) \setminus \mathbb{B}_R$ has six connected components:
\begin{eqnarray*}
\Gamma_{\pm 1} &:=& \left\{ (0, \pm t) \ | \ t \ge R \right\}, \\
\Gamma_{\pm 2} &:=& \left\{ (-2, \pm t) \ | \ t \ge R \right\}, \\
\Gamma_{\pm 3} &:=& \left\{ (\pm t, 0) \ | \ t \ge R \right\}.
\end{eqnarray*}
Consequently,
\begin{eqnarray*}
f|_{\Gamma_{\pm 1}} &=& -3t^2, \\
f|_{\Gamma_{\pm 2}} &=& - 8 - 3t^2, \\
f|_{\Gamma_{\pm 3}} &=& \pm t^3.
\end{eqnarray*}
It follows that $K = \{\pm 1, \pm 2, \pm 3\}$ and
\begin{eqnarray*}
\lambda_{\pm 1} &=& \lambda_{\pm 2} \ = \ -\infty \quad  \textrm{ and } \quad \lambda_{\pm 3} \ = \ \pm \infty.
\end{eqnarray*}
Therefore, by Theorem~\ref{Theorem31}, $f$ is bounded neither from below nor from above.
}\end{example}

\begin{example}{\rm
Let us consider the Motzkin polynomial (see \cite{Motzkin1967})
$$f({x}, {y}) := {x}^2{y}^4 + {x}^4 {y}^2 - 3 {x}^2{y}^2 + 1,$$
which is nonnegative on ${\Bbb R}^2.$ A simple calculation shows that
\begin{eqnarray*}
\Sigma(f, \mathbb{R}^2)  &=& \{{x} = 0\} \cup \{{y} = 0\} \cup \{(1, 1), (1, -1), (-1, 1), (-1, -1)\},
\end{eqnarray*}
and the tangency variety $\Gamma(f, \mathbb{R}^2)$ is given by the equation:
\begin{eqnarray*}
0 &=& \left( 4\,{{{x}}}^{3}{{ {y}}}^{2}+2\,{ {x}}\,{{ {y}}
}^{4}-6\,{ {x}}\,{{ {y}}}^{2} \right) { {y}}- \left( 2\,{{
 {x}}}^{4}{ {y}}+4\,{{ {x}}}^{2}{{ {y}}}^{3}-6\,{{
{x}}}^{2}{ {y}} \right) { {x}} \\
&=& {x} {y} \left({x}^2 - {y}^2\right) \left(6 - 2({x}^2 + {y}^2) \right).
\end{eqnarray*}
Hence, for $R > \sqrt{3},$ the set $\Gamma(f, \mathbb{R}^2) \setminus \mathbb{B}_R$ has eight connected components:
\begin{eqnarray*}
\Gamma_{\pm 1} &:=& \left\{ (\pm t, 0) \ | \ t \ge R \right\}, \\
\Gamma_{\pm 2} &:=& \left\{ (0, \pm t) \ | \ t \ge R \right\}, \\
\Gamma_{\pm 3} &:=& \left\{ (\pm t, \pm t) \ | \ t \ge R \right\}, \\
\Gamma_{\pm 4} &:=& \left\{ (\pm t, \mp t) \ | \ t \ge R \right\}.
\end{eqnarray*}
Consequently,
\begin{eqnarray*}
f|_{\Gamma_{\pm 1}} &=& \quad f|_{\Gamma_{\pm 2}} \ = \ 1, \\
f|_{\Gamma_{\pm 3}} &=& \quad f|_{\Gamma_{\pm 4}} \ = \ 2t^6 - 3t^4 + 1.
\end{eqnarray*}
It follows that $T_\infty(f, \mathbb{R}^2) = \{1 \},$ $K = \{\pm 3, \pm 4\},$ and
\begin{eqnarray*}
\lambda_{\pm 1} &=& \lambda_{\pm 2} \ = \ 1,\\
\lambda_{\pm 3} &=& \lambda_{\pm 4} \ = \ +\infty,\\
\alpha_{\pm 3} &=& \alpha_{\pm 4} \ = \ 6. 
\end{eqnarray*}
Therefore, in light of Theorems~\ref{Theorem31} and \ref{Theorem33}, $f$ is bounded from below and attains its infimum. By Corollary~\ref{Corollary31}, 
the set of optimal solutions of the problem $\inf_{(x, y) \in \mathbb{R}^2} f(x, y)$ is nonempty compact. In fact, we can see that this set is
$$f^{-1}(0) = \{(1, 1), (1, -1), (-1, 1), (-1, -1)\}.$$
Moreover, from Theorem~\ref{Theorem35} we have 
\begin{itemize}
\item[$\bullet$] If $\lambda < 0,$ then $\mathscr{L}(\lambda)$ is empty;
\item[$\bullet$] If $0 \le \lambda < 1,$ then $\mathscr{L}(\lambda)$ is nonempty compact;
\item[$\bullet$] If $\lambda > 1,$ then $\mathscr{L}(\lambda)$ is non-compact;
\item[$\bullet$] If $\lambda = 1,$ then the set $\mathscr{L}(\lambda)$ is non-compact because the function $\psi$ is constant  $1.$ In fact, $\mathscr{L}(1)$ contains the following unbounded set:
\begin{eqnarray*}
f^{-1}(1) &=& \{{x} = 0 \} \cup \{{y} = 0 \} \cup \{{x}^2 + {y}^2 = 3 \}.
\end{eqnarray*}
\end{itemize}
Finally, by Theorem~\ref{Theorem36}, the polynomial $f$ is not coercive.
}\end{example}

\begin{example}{\rm
Let $f(x, y) := (xy - 1)^2 + y^2$ be the polynomial considered in Example~\ref{Example21}.
Then the tangency variety $\Gamma(f, \mathbb{R}^2)$ is given by the equation:
\begin{eqnarray*}
2( -{x}^{3}y+x{y}^{3}+{x}^{2}-xy-{y}^{2}) \ = \ 0.
\end{eqnarray*}
We can see that\footnote{The computations are performed with the software Maple, using the command ``puiseux'' of the package ``algcurves'' for the rational Puiseux expansions.} for $R$ large enough, the set $\Gamma(f, \mathbb{R}^2) \setminus \mathbb{B}_R$ has eight connected components:
\begin{eqnarray*}
\Gamma_{\pm 1}: && x := t, \qquad y := -t + \frac{1}{2} t^{-1} + \frac{5}{8} t^{-3} + \cdots,\\
\Gamma_{\pm 2}: && x := t, \qquad y :=  t + \frac{1}{2} t^{-1} + \frac{3}{8} t^{-3} + \cdots,\\
\Gamma_{\pm 3}: && x := t, \qquad y :=  t^{-1} - t^{-3} + \cdots,\\
\Gamma_{\pm 4}: && x := t^{-1}, \quad y :=  t + t^{-1} - t^{-3} + \cdots,
\end{eqnarray*}
where $t \to \pm \infty.$ Then substituting these expansions in $f$ we get
\begin{eqnarray*}
f|_{\Gamma_{\pm 1}} &=& t^4 + 4t^2 + 2 - \frac{23}{8}t^{-2} + \cdots,\\
f|_{\Gamma_{\pm 2}}  &=& t^4 + 2 + \frac{5}{8}t^{-2} + \cdots, \\
f|_{\Gamma_{\pm 3}}  &=& t^{-2} - t^{-4} + t^{-6} + \cdots, \\
f|_{\Gamma_{\pm 4}}  &=& t^2 + 2 - t^{-2} - t^{-4} + \cdots.
\end{eqnarray*}
It follows that $T_\infty(f, \mathbb{R}^2) = \{0 \},$ $K = \{\pm 1, \pm 2, \pm 3, \pm 4\},$ and
\begin{eqnarray*}
\lambda_{\pm 1} &=& \lambda_{\pm 2} \ = \ \lambda_{\pm 4}  \ = \ +\infty, \quad \lambda_{\pm 3}  \ = \ 0.
\end{eqnarray*}
In light of Theorem~\ref{Theorem31}, $f$ is bounded from below. Note that $\Sigma(f, \mathbb{R}^2) = \{(0, 0)\}$ and
\begin{eqnarray*}
f(0, 0) & = & 1 \ > \ 0 \ = \ \min_{k = 1, 2, 3, 4} \lambda_{\pm k}.
\end{eqnarray*}
Hence, by Theorem~\ref{Theorem33}, $f$ does not attain its infimum. Furthermore, in view of Corollary~\ref{Corollary22}, we have
\begin{eqnarray*}
\inf_{(x, y) \in \mathbb{R}^2} f(x, y) &=& 0.
\end{eqnarray*}
}\end{example}

\subsection*{Acknowledgments}
The author wishes to thank J\'er\^ome Bolte for the useful discussions. The last version of this paper was partially performed while the author had been visiting
at the Vietnam Institute for Advanced Study in Mathematics (VIASM) from January 1 to 31 March, 2019. He would like to thank the Institute for hospitality and support.

\bibliographystyle{abbrv}

\end{document}